\documentclass[12pt]{amsart}
\author{Tamanna Chatterjee}
\date{\today}
\title{A Note on Primitive Pairs for Graded Lie Algebras}

\usepackage{amssymb, amsmath, amsfonts, amsthm, mathrsfs}
\usepackage{calligra, mathtools, tikz-cd, pgfkeys, pgfopts, xcolor}
\usepackage{ytableau, placeins, mathrsfs, geometry}
\usetikzlibrary{matrix}

\theoremstyle{plain}
\newtheorem{theorem}{Theorem}[section]

\newtheorem{lemma}[theorem]{Lemma}
\newtheorem{pro}[theorem]{Proposition}

\theoremstyle{definition}
\newtheorem{definition}[theorem]{Definition}
\newtheorem{remark}[theorem]{Remark}

\newcommand{\mf}{\mathfrak}
\newcommand{\mc}{\mathcal}
\newcommand{\mb}{\mathbb}
\newcommand{\ms}{\mathscr}

\DeclareMathOperator{\p}{Perv}
\DeclareMathOperator{\f}{For}

\DeclareMathOperator{\rhm}{R\mathscr{H}\text{\kern -3pt{\calligra\large om}}\, }
\DeclareMathOperator{\lie}{Lie}
\DeclareMathOperator{\h}{H}

\DeclareMathOperator{\cu}{cusp}

\DeclareMathOperator{\ind}{^{\mathfrak{g},\mathfrak{p}}\mathcal{I}}
\DeclareMathOperator{\inn}{Ind}
\DeclareMathOperator{\rnn}{Res}

\DeclareMathOperator{\i'}{\mathscr{I}}
\DeclareMathOperator{\phign}{\Phi_{\mathfrak{g}_n}}
\begin{document}

\maketitle

\begin{abstract}
We develop a theory of \emph{primitive pairs} for $\mathbb{Z}$-graded Lie algebras when the sheaves have coefficients in a field $\Bbbk$  of positive characteristic, providing a graded analogue of the role played by cuspidal pairs in the generalized Springer correspondence. We consider the centralizer $G_0$ of a fixed cocharacter $\chi$ in a connected, reductive, algebraic group $G$ and its action on the eigenspaces $\mf{g}_n$ of $\chi$. Building on the framework of parity sheaves and the Fourier transform established in \cite{Ch,Ch1}, we show that every indecomposable parity sheaf  on  $\mathfrak{g}_n$  can be expressed as a direct summand of a complex induced from primitive data on a Levi subgroup. This result extends the fact that,  in the graded setting, any indecomposable parity sheaf is direct summand of an induced cuspidal datum \cite{Ch}. This confirms the organizing role of primitive pairs in the block decomposition of the category of $G_0$-equivariant parity sheaves on $\mathfrak{g}_n$. We further establish that primitive pairs on the nilpotent cone induce primitive pairs in the graded setting, and we prove that primitivity is preserved under the Fourier--Sato transform. These results reveal a deep compatibility between the geometry of graded Lie algebras and their representation-theoretic structures, forming the foundation for a graded version of the generalized Springer correspondence in positive characteristic.

\end{abstract}

\section*{Introduction}

The generalized Springer correspondence provides a geometric framework connecting the representation theory of reductive groups with the geometry of their nilpotent cones. Let $G$ be a complex, connected, reductive algebraic group, and let $\mathfrak{g}$ denote its Lie algebra. In both characteristic zero and positive characteristic, the category 
$\mathscr{I}(G)$
of irreducible $G$-equivariant perverse sheaves on the nilpotent cone admits a block decomposition, in which \emph{cuspidal pairs} play a fundamental organizing role. Cuspidal pairs serve as the basic building blocks: every simple perverse sheaf arises via parabolic induction from cuspidal data on a Levi subgroup \cite{AJHR,AJHR2,AJHR3,AHJR4}.

In the $\mathbb{Z}$-graded setting, the analogous notion is that of \emph{primitive pairs}.  We fix a cocharacter $\chi: \mathbb{C}^\times \to G$. The group $\mathbb{C}^\times$ acts on $\mathfrak{g}$ via the adjoint action, and the space $\mathfrak{g}_n$ represents the $n$-th weight space under this action, thereby defining a grading on $\mathfrak{g}$. The centralizer $G_0$ of $\chi(\mathbb{C}^\times)$ also acts on $\mathfrak{g}_n$ via the adjoint action. 

The $G_0$-equivariant derived category of sheaves with coefficients in a field $\Bbbk$, $D^b_{G_0}(\mf{g}_n,\Bbbk)$ has been studied in \cite{Lu} when the characteristic of $\Bbbk$ is $0$ and  in \cite{Ch, Ch1} when the characteristic is positive. In positive characteristic, parity sheaves, introduced in \cite{parity} play  a similar role to the intersection cohomology complexes ($\mc{IC}$'s) in characteristic $0$.  For any pair $(\mc{O},\mc{L})$, where $\mc{O}$ is a $G_0$-orbit in $\mf{g}_n$ and $\mc{L}$ is a $G_0$-equivariant irreducible local system on $\mc{O}$, there exists at most one indecomposable parity sheaf $\mc{E}$ up to shift with the property that $\mc{E}|_{\mc{O}}=\mc{L}[\dim \mc{O}]$, and we write $\mc{E}(\mc{O},\mc{L})$ to denote $\mc{E}$.   We will denote the collection of all such pairs $(\mc{O},\mc{L})$ by $\ms{I}(\mf{g}_n)$. Similarly, on the nilpotent cone, for each pair $(C,\mc{F})$, where $C$ is a $G$-orbit in the nilpotent cone $\mc{N}$ and $\mc{F}$ is a $G$-equivariant irreducible local system on $C$, there exists at most one parity sheaf, denoted by $\mc{E}(C,\mc{F})$ with the property  $\mc{E}(C,\mc{F})|_{C}=\mc{F}[\dim C] $.

A primitive pair on both nilpotent cone and graded Lie algebras arises from a cuspidal pair defined on the nilpotent cone or graded Lie algebra of a Levi subgroup (Definition \ref{def4.1} and \ref{def4.2}). Lusztig introduced this concept in \cite{Lu}, showing that  when the sheaf coefficients are coming from the field $\Bbbk$ of characteristic $0$, the set $
\mathscr{I}(\mathfrak{g}_n)$ decomposes into subsets $\mathscr{I}(G,L,\mathcal{L})$,
where $L$ is a Levi subgroup and $\mathcal{L}$ is a primitive local system supported on an nilpotent $L$-orbit. This provides a graded counterpart to the generalized Springer correspondence. Similar phenomena were later observed in the work of Yun and Lusztig \cite{LY} on $\mathbb{Z}/m$-graded Lie algebras, where primitive pairs again play a central role in the block decomposition of 
$\mathscr{I}(\mathfrak{g}_i)$, where $\mf{g}_i$ is a $\mb{Z}/m$-graded piece of $\mf{g}$ and $\mathscr{I}(\mathfrak{g}_i)$ denote the collection of $G_{\underbar{0}}$-equivariant irreducible local system on $\mf{g}_i$ with $G_{\underbar{0}}$ being an appropriate  analogue of $G_0$.
Motivated by these developments, one expects a parallel theory for $\mathbb{Z}$-gradings when the sheaf coefficients lie in a field $\Bbbk$ of positive characteristic.

This paper is a continuation of \cite{Ch,Ch1}, where parity sheaves and the Fourier--Sato transform on graded Lie algebras were studied in positive characteristic. The assumptions \cite[Assumption ~2.4]{Ch1} and  conjectures \cite[Conjectures ~2.8, 2.9]{Ch1} stated there are also assumed to hold in the present setting. Here we prove: 
\begin{theorem}
	Any parity sheaf
$\mathcal{E}(\mathcal{O},\mathcal{L}) \in D^b_{G_0}(\mathfrak{g}_n)$
can be expressed as
 $\mathcal{E}(\mathcal{O},\mathcal{L}) \oplus \mathcal{E}_1 \cong \mathcal{E}_2,$
where $\mathcal{E}_1$ and $\mathcal{E}_2$ are parity complexes induced from some primitive pairs.
\end{theorem}
 This result generalizes Theorem~36 of \cite{Ch} and can be viewed as the graded analogue of the role played by induced cuspidal perverse sheaves in the generalized Springer correspondence. In that correspondence, cuspidal pairs generate the entire category through parabolic induction; in our setting, primitive pairs serve the same purpose for graded Lie algebras.

We also prove that a primitive pair on the nilpotent cone gives rise to a primitive pair in the graded setting under suitable assumptions. Furthermore, we study how primitive pairs behave under the Fourier--Sato transform $\phign$, which is a functor from $D^b_{G_0}(\mf{g}_n)$ to $D^b_{G_0}(\mf{g}_{-n})$, as developed in \cite{Ch1}. In particular, we establish the following key statement:

\begin{theorem}
Let $(\mathcal{O},\mathcal{L}) \in \mathscr{I}(\mathfrak{g}_n)$, where $\mathcal{O}$ is the unique open $G_0$-orbit in $\mathfrak{g}_n$, and suppose
$\Phi_{\mathfrak{g}_n}\mathcal{E}(\mathcal{O},\mathcal{L})[\dim \mathfrak{g}_n] \cong \mathcal{E}(\mathcal{O}',\mathcal{L}')[\dim \mathfrak{g}_{-n}]$, with $(\mathcal{O}',\mathcal{L}') \in \mathscr{I}(\mathfrak{g}_{-n})$ and $\mathcal{O}'$ the unique open $G_0$-orbit in $\mathfrak{g}_{-n}$.
Then $(\mathcal{O},\mathcal{L})$ is a primitive pair.
\end{theorem}

This result demonstrates that primitivity is preserved under the Fourier transform, highlighting a deep compatibility between the geometric and graded structures.

\section*{Organization}
The paper is organized as follows. Section~1 recalls the necessary background, notations, and assumptions from \cite{Ch,Ch1}. Section~2 introduces primitive pairs for graded Lie algebras and proves that primitive pairs on the nilpotent cone induce primitive pairs in the graded setting. Section~3 studies their geometric properties, including $n$-rigidity. Section~4 analyzes the interaction of primitive pairs with the Fourier--Sato transform and proves Proposition~4.2. Finally, Section~5 introduces quasi-monomial and good objects and establishes a decomposition theorem expressing any parity sheaf as a direct summand of those induced from primitive data.

\section{Background}\label{background}

Let $\Bbbk$ be a field of characteristic $\ell\geq 0$. We  consider  sheaves with coefficients in $\Bbbk$. Here we briefly recall some definitions and results from \cite{Ch} and \cite{Ch1}.

Let  $G$  be a connected, reductive, algebraic group over 
$\mathbb{C}$ and let $\mathfrak{g}$ be the Lie algebra of $G$. 
We  fix a cocharacter $\chi:\mathbb{C}^\times \to G$ and define: \[G_0=\{g\in G| g\chi(t)=\chi(t)g, \forall{t}\in \mathbb{C}^\times\}.\]
 For $n\in \mathbb{Z}$, define: \[\mathfrak{g}_n=\{x\in \mathfrak{g}|\hspace{1mm} \mathrm{Ad}(\chi(t))x=t^nx, \forall{t}\in \mb{C}^\times\}.\] This defines a grading on $\mathfrak{g}$,
 \[\mathfrak{g}=\bigoplus_{n\in \mathbb{Z}}\mathfrak{g}_n.\]Clearly, $\mathfrak{g}_0=\lie(G_0)$ and
 $G_0$ acts on $\mathfrak{g}_n$. 	For $n\neq0$, $G_0$ acts on $\mathfrak{g}_n$ with only finitely many orbits.
We will consider the $G_0$-equivariant bounded derived category of constructible sheaves on $\mf{g}_n$, $D^b_{G_0}(\mf{g}_n,\Bbbk)$ or $D^b_{G_0}(\mf{g}_n)$. We follow the notations from \cite{Ch} and \cite{Ch1}. Let $\mc{N}_G$ denote the nilpotent cone of $G$.  The irreducible perverse sheaves on $\mc{N}_G$ are in bijection with the set \[\ms{I}(G)=\{(C,\mc{E})| C \text{ is a $G$-orbit in }\mc{N}_G \text{ and $\mc{E}$ is a local system on } C \}.\] similarly the irreducible perverse sheaves on $\mf{g}_n$ are in bijection with \[\ms{I}(\mf{g}_n)=\{(\mc{O},\mc{L})| \mc{O} \text{ is a $G_0$-orbit in }\mf{g}_n \text{ and $\mc{L}$ is a local system on } \mc{O} \}.\]   To use the results from \cite{Ch} and \cite{Ch1} we make the assumption that $\ell$  is `Pretty good' and `big enough' for $G$  \cite[Assumption ~7]{Ch}. We also assume Conjecture 2.8 and Conjecture 2.9 from \cite{Ch1} are true in our set-up.

Let $P$ be a parabolic subgroup of $G$ with unipotent radical $U_P$ and let $L\subset P$ be a Levi factor of $P$. One can identify $L$ with $P/U_P$ through the natural morphism, $L\xhookrightarrow{}P \twoheadrightarrow{} P/U_P$. As  defined in \cite{AJHR2}, we have a functor $\rnn^G_P: D^b_G(\mc{N}_G)\to D^b_L(\mc{N}_L)$.  
 A simple object $\mathcal{F}$ in $\p_G(\mathscr{N}_G,\Bbbk)$ is called cuspidal if $\rnn^G_P(\mc{F})=0$, for any proper parabolic $P$ and  Levi factor $L\subset P$.
	 A pair $(C,\mathcal{E})\in \mathscr{I}(G)$, is called cuspidal if the corresponding simple perverse sheaf $\mathcal{IC}(C,\mathcal{E})$ is cuspidal. We will denote the collection of all cuspidal pairs on $\mc{N}_G$ by $\ms{I}(G)^{\cu}$. 
	 Using modular reduction map \cite[~ 2.3]{AJHR}, we can define $0$-cuspidal and we will denote this collection by $\ms{I}(G)^{0-\cu}$ and $\ms{I}(G)^{0-\cu}\subset \ms{I}(G)^{\cu}$ \cite[Lemma ~2.3]{AJHR}.
	  A brief discussion on cuspidal pairs and $0$-cuspidal pairs can also be found in \cite[Section ~2]{Ch}.

\begin{definition}
A pair $(\mathcal{O},\mathcal{L})\in \mathscr{I}(\mathfrak{g}_n,\Bbbk)$ will be called  cuspidal if there exists a pair $(C,\mc{E})\in \ms{I}(G)^{0-\cu}$, such that $C\cap \mf{g}_n=\mc{O}$ and $\mc{L}=\mc{E}|_{\mc{O}}$.
  We will denote the set of all cuspidal pairs on $\mf{g}_n$ by 
$\mathscr{I}(\mathfrak{g}_n)^{\cu}$.
\end{definition}

\begin{definition}[The set $J_n$]
 We define,
\[
J_n=\{\varphi:\mathfrak{sl}_2\to\mathfrak g \mid 
\varphi(e)\in\mathfrak g_n,\ \varphi(f)\in\mathfrak g_{-n},\ \varphi(h)\in\mathfrak g_0\}.
\]
\end{definition}

\begin{definition}[{\em $n$-rigid} (\cite{Lu})]
\label{def1.3}
The pair $(G,\chi)$ is called \emph{$n$-rigid} if there exists a Lie algebra
homomorphism $\varphi:\mathfrak{sl}_2\to\mathfrak g$ satisfying:
\begin{enumerate}
\item $\varphi\in J_n$;
\item\label{item:weights1} if we write the $\mathbb C^\times$--cocharacter
      $\widetilde\varphi:\mathbb C^\times\to G$ induced by $\varphi$ and set
      for each $m\in\mathbb Z$
      \[
      \mathfrak m_m:=\{X\in\mathfrak g\mid
      \operatorname{Ad}(\widetilde\varphi(t))X=t^m X\},
      \]
      then
      \[
      \mathfrak m_m=\mathfrak g_{\,\frac{n m}{2}}\text{, if }\frac{n m}{2}\in\mathbb Z;
      \]
\item\label{item:weights2} and
      \[
      \mathfrak m_m=0\text{, if }\frac{n m}{2}\notin\mathbb Z.
      \]
\end{enumerate}
\end{definition}

\begin{definition}[{\em $n$-adapted} \cite{Lu}]\label{def:n-adapted}
A homomorphism $\varphi:\mathfrak{sl}_2\to\mathfrak g$ (equivalently its
associated cocharacter $\widetilde\varphi:\mathbb C^\times\to G$) is said to
be \emph{$n$-adapted} to the grading cocharacter $\chi$ if the following hold:
\begin{itemize}
\item $\varphi\in J_n$ (so $\varphi(e)\in\mathfrak g_n$, etc.);
\item the cocharacters $\chi$ and $\widetilde\varphi$ commute (hence one has a
      simultaneous bigrading of $\mathfrak g$ by the two $\mathbb C^\times$--actions);
\item the weight spaces for $\widetilde\varphi$ align with the $\chi$--grading
      in the sense of items (\ref{item:weights1})--(\ref{item:weights2}) above:
      i.e. the $\widetilde\varphi$--weight $m$ subspace equals
      $\mathfrak g_{nm/2}$ when $nm/2\in\mathbb Z$, and is zero otherwise.
\end{itemize}
(When such a $\varphi$ exists one often says ``$\varphi$ is $n$-adapted to~$\chi$''.)
\end{definition}

\section{Primitive Pair} \label{sec5}

Primitive pairs were introduced by Lusztig \cite{Lu} both on the nilpotent cone and on the graded Lie algebra. One of the main   goals of that paper was to define a block decomposition of $\i'(\mf{g}_n)$ using primitive pairs on the nilpotent cone. In this section, we aim to establish positive characteristic analogues of several results from \cite{Lu}. A long-term goal is to study the block decomposition of $\ms{I}(\mf{g}_n)$ in positive characteristic.

A primitive pair on both nilpotent cone and graded Lie algebras arises from a cuspidal pair defined on the nilpotent cone or graded Lie algebra of a Levi subgroup. The definition below formalizes this idea by requiring the existence of an appropriate parabolic subgroup $P$ and a graded version of the compatibility conditions that characterize primitivity on the nilpotent cone.

\begin{definition}\label{def4.1}
A pair $(C,\mc{E})\in \ms{I}(G)$ is said to be \emph{primitive} if there exists a parabolic subgroup $P$ with a Levi factor $L$ and a cuspidal pair $(C_L,\mc{E}_L)\in\i'(L)^{0-\cu}$ such that the following hold:
\begin{enumerate}
    \item There exists a $P$-orbit $C'$ in $\mc{N}_P$ such that for all $x\in C'$, we have $\pi(x)\in C_L$, where $\pi: \mf{p}\to \mf{l}$.
    \item For any $x\in C'$, we have\[L^{\pi(x)}/(L^{\pi(x)})^\circ \cong P^x/(P^x)^\circ \cong G^x/(G^x)^\circ.
    \]
    \item $\mc{E}|_{C'}\cong (\pi|_{C'})^* \mc{E}_L$.
\end{enumerate}
\end{definition}

The definition of a primitive pair on $\mf{g}_n$ is analogous to the one above for the nilpotent cone.

\begin{definition}\label{def4.2}
A pair $(\mc{O},\mc{L})\in \ms{I}(\mf{g}_n)$ is said to be \emph{primitive} if there exists a parabolic subgroup $P$ containing the image of $\chi$ (defined in Section \ref{background}), with Levi factor $L$, and a cuspidal pair $(\mc{O}_L,\mc{E}_L)\in \i'(\mf{l}_n)^{\cu}$ such that the following hold:
\begin{enumerate}
    \item There exists a $P_0$-orbit $\mc{O}'$ in $\mf{p}_n$ such that for all $x\in \mc{O}'$, we have $\pi(x)\in \mc{O}_L$, where $\pi: \mf{p}_n\to \mf{l}_n$.
    \item For any $x\in \mc{O}'$, we have
    \[
    L_0^{\pi(x)}/(L_0^{\pi(x)})^\circ \cong P_0^x/(P_0^x)^\circ \cong G_0^x/(G_0^x)^\circ.
    \]
    \item $\mc{L}|_{\mc{O}'}\cong (\pi|_{\mc{O}'})^* \mc{E}_L$.
\end{enumerate}
\end{definition}

\begin{theorem}\label{th5.3}
Assume $(C,\mc{E})\in \ms{I}(G)$ is a primitive pair with $C\cap \mf{g}_n\neq \emptyset$, associated with a parabolic subgroup $P$ containing a Levi factor $L$ and a cuspidal pair $(C_L,\mc{E}_L)\in \ms{I}(L)^{0-\cu}$ such that $\chi(\mb{C}^\times)\subset L$ and $C_L\cap \mf{l}_n\neq \emptyset$. Then $(C,\mc{E})\in \ms{I}(G)$ gives rise to a primitive pair $(\mc{O},\mc{L})\in \ms{I}(\mf{g}_n)$.
\end{theorem}

\begin{proof}
Let $C'$ be the $P$-orbit in $\mc{N}_P$ satisfying the conditions in Definition~\ref{def4.1}. Let $\mc{O}_L=C_L\cap \mf{l}_n$ and $\mc{F}_L=\mc{E}_L|_{\mc{O}_L}$. Then, by the definition of a cuspidal pair in the graded setting, $(\mc{O}_L,\mc{F}_L)\in \ms{I}(\mf{g}_n)^{\cu}$. Define
\[
\mc{O}'=\{x\in \mf{p}_n \mid \text{there exists a Levi } M'\subset P \text{ with } \chi(\mb{C}^\times)\subset M',\, x\in \mf{m}'_n,\, \pi(x)\in \mc{O}_L\}.
\]
Since all Levi subgroups inside $P$ are $P_0$-conjugate, $\mc{O}'$ is a single $P_0$-orbit in $\mf{p}_n$. Let $\mc{O}$ be the unique $G_0$-orbit in $\mf{g}_n$ containing $\mc{O}'$, and set $\mc{L}=\mc{E}|_\mc{O}$. We claim that $(\mc{O},\mc{L})$ is primitive.

By construction, $\mc{O}'$ satisfies the first condition in Definition~\ref{def4.2}. Since $(C,\mc{E})\in \ms{I}(G)$ is primitive, for $x\in \mc{O}'$ we have
\[
L^{\pi(x)}/(L^{\pi(x)})^\circ \cong P^x/(P^x)^\circ \cong G^x/(G^x)^\circ.
\]
First, we show that $\mc{L}$ is irreducible. It suffices to show that the map $G_0^x/(G_0^x)^\circ \to G^x/(G^x)^\circ$ is surjective.

We have the following commutative diagram:
\[
\begin{tikzcd}
L_0^{\pi(x)}/(L_0^{\pi(x)})^\circ \arrow[d] & P_0^x/(P_0^x)^\circ \arrow[l] \arrow[d]\arrow[r] & G_0^x/(G_0^x)^\circ \arrow[d] \\ 
L^{\pi(x)}/(L^{\pi(x)})^\circ  & P^x/(P^x)^\circ \arrow[l] \arrow[r] & G^x/(G^x)^\circ
\end{tikzcd}
\]
The lower horizontal maps are isomorphisms by Definition~\ref{def4.1}. The map $P_0^x\to L_0^{\pi(x)}$ has connected kernel, hence induces an isomorphism on component groups. By \cite[Prop.~4.3]{Lu}, $(L,\chi)$ is $n$-rigid, and therefore by \cite[Prop.~4.2(c)]{Lu},
\begin{equation}\label{4.1}
L^{\pi(x)}/(L^{\pi(x)})^\circ\cong L_0^{\pi(x)}/(L_0^{\pi(x)})^\circ.
\end{equation}
Hence, the middle vertical arrow is an isomorphism, implying that in the right square, the composition of the top horizontal and right vertical maps is an isomorphism. This shows that the map $G_0^x/(G_0^x)^\circ \to G^x/(G^x)^\circ$ is surjective, proving that $\mc{L}$ is irreducible.

To show that $(\mc{O},\mc{L})$ is primitive, we still need isomorphisms for the upper horizontal maps.

By the definition of $\mc{O}'$, there exists a Levi subgroup $M'\subset P$ with $\chi(\mb{C}^\times)\subset M'$, $x\in \mf{m}'_n$, and $\pi(x)\in \mc{O}_L$. We can find a Lie algebra homomorphism $\phi: \mf{sl}_2\to \mf{m}'$ with $\phi(e)=x$, $\phi(f)\in \mf{g}_{-n}$, and $\phi(h)\in \mf{g}_0$, where $\{e,h,f\}$ is the standard $\mf{sl}_2$-triple. Let $\tilde{\phi}:SL_2\to M'$ be such that $d\tilde{\phi}=\phi$, and define $\chi': \mb{C}^\times \to M'$ by
\[
\chi'(a)=\tilde{\phi}\begin{pmatrix}
a & 0 \\ 0 & a^{-1}
\end{pmatrix}.
\]
Then $\phi$ satisfies the conditions in \cite[Definition~6.1]{Ch} with respect to $\chi:\mb{C}^\times\to M'$. Viewing $\chi$ as a map to $G$, we use the construction in \cite[§6.2]{Ch} to define the parabolic and Levi subalgebras $\mf{q}$ and $\mf{m}$. It is easy to see that $x\in \mf{m}$. By \cite[Theorem.~6.3]{Ch}, $(M,\chi)$ is $n$-rigid, and hence
\begin{equation}\label{4.2}
M^{x}/(M^{x})^\circ\cong M_0^{x}/(M_0^{x})^\circ, \quad \text{by \cite[Prop.~4.2(c)]{Lu}}.
\end{equation}
Also, by \cite[Prop.~5.8]{Lu},
\begin{equation}\label{4.3}
M_0^{x}/(M_0^{x})^\circ\cong G_0^{x}/(G_0^{x})^\circ.
\end{equation}
As $M'$ is a Levi subgroup of $P$ conjugate to $L$, we have from Definition~\ref{def4.1},
\begin{equation}\label{4.4}
M'^x/(M'^x)^\circ \cong G^x/(G^x)^\circ.
\end{equation}
Since $\phi:\mf{sl}_2\to \mf{m}'$ is $n$-adapted to $\chi$, we have
\[
{}_{m}\mf{m}'=\mf{m}'_{\frac{mn}{2}} \quad \text{if } \tfrac{mn}{2}\in \mb{Z}.
\]
By the construction in \cite[§6.2]{Ch}, $\mf{m}=\bigoplus_{m} ({}_{m}\mf{g}\cap \mf{g}_{\frac{mn}{2}})$ when $\frac{mn}{2}\in \mb{Z}$. Therefore $\mf{m}'\subset \mf{m}$. Hence, the triple $(M,M',x)$ plays the same role as $(G,M',x)$, and we obtain
\begin{equation}\label{4.5}
M'^x/(M'^x)^\circ \cong M^x/(M^x)^\circ.
\end{equation}
Combining \eqref{4.2}, \eqref{4.3}, \eqref{4.4}, and \eqref{4.5}, we get
\[
G_0^x/(G_0^x)^\circ \cong G^x/(G^x)^\circ.
\]

In the commutative diagram above, the top left horizontal map is already an isomorphism. We now show that the top right horizontal map is also an isomorphism. Since the left and right vertical and lower horizontal maps in the right square are all isomorphisms, the top horizontal map must be an isomorphism as well.

Finally, consider the following commutative diagram:
\[
\begin{tikzcd}
C_L & C' \arrow[l, "\pi|_{C'}"'] \arrow[r, hook] & C\\ 
\mc{O}_L \arrow[u, hook'] & \mc{O}' \arrow[u, hook'] \arrow[l, "\pi|_{\mc{O}'}"'] \arrow[r, hook] & \mc{O} \arrow[u, hook']
\end{tikzcd}
\]
Then
\begin{align*}
\mc{L}|_{\mc{O}'} 
&\cong (\mc{E}|_{C'})|_{\mc{O}'} && \text{by the right commutative square},\\
&\cong ((\pi|_{C'})^* \mc{E}_L)|_{\mc{O}'} && \text{by the definition of a primitive pair on $G$},\\
&\cong (\pi|_{\mc{O}'})^* \mc{E}_L|_{\mc{O}_L} && \text{by the left commutative square},\\
&\cong (\pi_{\mc{O}'})^* \mc{F}_L.
\end{align*}
Hence, the third condition in Definition~\ref{def4.2} is satisfied.
\end{proof}

The next theorem appears in \cite{Lu}; its proof is similar in all characteristics.

\begin{lemma}\label{lm5.4}
Assume $(\mc{O},\mc{L})\in \i'(\mf{g}_n)$ is a primitive pair, and let 
\[
P,L,(\mc{O}_L,\mc{L}_L)\in \i'(\mf{l}_n)^{\cu} \quad \text{and} \quad P',L',(\mc{O}_{L'},\mc{L}_{L'})\in \i'(\mf{l}'_n)^{\cu}
\]
be two data consisting of parabolics, Levi subgroups, and cuspidal pairs associated to $(\mc{O},\mc{L})$. Then there exists $g\in G_0$ such that 
\[
P',L',(\mc{O}_{L'},\mc{L}_{L'}) = g\cdot (P,L,(\mc{O}_L,\mc{L}_L)).
\]
\end{lemma}

The proof follows from \cite[§11.4]{Lu}.

\section{More on Primitive Pairs}\label{subsec5.1}

We now fix the following assumptions:
\begin{itemize}
	\item A primitive pair $(C, \mc{F}) \in \ms{I}(G)$ such that $C \cap \mf{g}_n \neq \emptyset$.
	\item A $G_0$-orbit $\mc{O} \subset \mf{g}_n$ such that $\mc{O} \subset C$.
	\item An element $y \in \mc{O}$ and $\phi \in J_n$ such that $\phi(e) = y$.
	\item A maximal torus $S$ of $(G_0^\phi)^\circ$, where $G_0^\phi = \{ g \in G_0 \mid \mathrm{Ad}(g)x = x \text{ for all } x \in \mathrm{im} \, \phi \}$.
\end{itemize}

Let $\mc{L} = \mc{F}|_\mc{O}$ and set $L = Z_G(S)$. This implies $\chi(\mb{C}^\times) \subset L$.
The proof of the Proposition below is similar to the one in \cite{Lu} and is the same in all characteristics.
\begin{pro}\label{prop5.4}
Under the above assumptions, the following statements hold:
\begin{enumerate}
	\item $\mc{L}$ is an irreducible local system.
	\item Let $\mc{O}_L$ be the unique open $L_0$-orbit in $\mf{l}_n$. Then $\mc{O}_L \subset \mc{O}$ and $(\mc{O}_L, \mc{L}|_{\mc{O}_L}) \in \i'(\mf{l}_n)^{\cu}$.
	\item For any parabolic subgroup $P$ containing $L$, the pair $(\mc{O}, \mc{L})$ is the primitive pair associated to $(P, L, \mc{O}_L, \mc{L}|_{\mc{O}_L})$.
\end{enumerate}
\end{pro}

\begin{proof}
\begin{enumerate}
	\item This follows from Theorem~\ref{th5.3}.
	\item Since $S$ is a maximal torus in $(G_0^\phi)^\circ$, it is also a torus in $(G^\phi)^\circ$. We claim that $S$ is maximal in $(G^\phi)^\circ$. Indeed, for any maximal torus $T$ in $(G^\phi)^\circ$, it commutes with $\chi(\mb{C}^\times)$, as $\chi(\mb{C}^\times)$ normalizes $(G^\phi)^\circ$. Therefore $T$ is a torus in $(G_0^\phi)^\circ$, hence $T \subset S$. 

By \cite[4.2(d)(e)]{Lu}, $(G^y)^\circ$ admits $(G^\phi)^\circ$ as a Levi factor, so $S$ is also a maximal torus in $(G^y)^\circ$. Let $Q$ be the parabolic subgroup of $G$ with Levi subgroup $M$ and cuspidal pair $(C_M, \mc{F}_M) \in \i'(M)^{0-\cu}$ associated to the primitive pair $(C, \mc{F})$. Since $y \in \mc{O} \subset C$, by conjugating $Q$ by an element of $G$ we may assume $y \in \mf{m}'$ for a Levi subgroup $M'$ of $Q$, and that $\pi(y) \in C_M$. By \cite[2.7(f)]{Lu}, $M'$ is the centralizer of a maximal torus $T$ of $(G^y)^\circ$. Thus $T$ is conjugate to $S$ by an element of $(G^y)^\circ$. Conjugating $Q, M'$ by this element, we may assume $S = T$. As $L = Z_G(S)$, it follows that $L = M'$. Hence $L$ is conjugate to $M$ in $P$.

Let $(C_L, \mc{F}_L) \in \i'(L)^{0-\cu}$ be the cuspidal pair corresponding to $(C_M, \mc{F}_M)$. Since $y \in \mf{m}'$, we have $y \in \mf{l}$. Moreover, as $y \in \mf{g}_n$, it follows that $y \in \mf{l}_n$. Combining $y \in \mf{l}$ and $\pi(y) \in C_M$, we deduce $y \in C_L$. Thus $y$ is distinguished, which implies $y \in \mc{O}_L$ by \cite[4.3]{Lu}. Hence $\mc{O}_L \cap \mc{O} \neq \emptyset$. 

Now $\mc{O}_L$ is an $L_0$-orbit, and $\mc{O}$ is a $G_0$-orbit, so $L_0$ acts on $\mc{O}$ and $L_0 \cdot y = \mc{O}_L \subset \mc{O}$. Setting $\mc{L}_L = \mc{F}_L|_{\mc{O}_L}$, the definition of a cuspidal pair in the graded setting gives $(\mc{O}_L, \mc{L}_L) \in \i'(\mf{l}_n)^{\cu}$.
	\item This follows from Theorem~\ref{th5.3}.
	
\end{enumerate}
\end{proof}

\subsection{{$n$}-Rigidity}

Recall the definition of $n$-rigidity from \ref{def1.3}. In this subsection, we retain all assumptions from the beginning of section \ref{subsec5.1}.  
Additionally, we assume that $(G, \chi)$ is $n$-rigid and that the unique open $G_0$-orbit $\mc{O}_n$ in $\mf{g}_n$ is contained in $C$. Then, by \cite[4.2(f)]{Lu}, we have $C \cap \mf{g}_n = \mc{O}_n$, hence $\mc{O} = \mc{O}_n$. Consequently, $\phi(e) = y \in \mc{O}_n$ by \cite[4.2(a)]{Lu}.  

Since $(G, \chi)$ is $n$-rigid, $\phi$ is $n$-adapted to $\chi$, and by \cite[4.2(e)]{Lu}, $G_0^y = G^\phi$. Therefore, as $G_0^y \subset G_0^\phi$, it follows that $G^\phi \subset G_0^\phi$, hence $G^\phi = G_0^\phi$.  

By \cite[Prop.~11.7]{Lu}, for any parabolic subgroup $P$ of $G$, the intersection $(G^\phi)^{\circ} \cap P$ is a Borel subgroup of $(G^\phi)^{\circ}$. Therefore, $(G_0^y)^{\circ} \cap P$ is a Borel subgroup of $(G_0^y)^{\circ}$.

\begin{lemma}\label{lm4.6}
For $P, L, U$ as defined above, the $P_0$-orbit of $y$ equals $\mc{O}_L + \mf{u}_n$.
\end{lemma}

The proof follows from \cite[Prop.~11.10]{Lu}.

\begin{lemma}\label{lm4.5}
Let $x \in \mathcal{O}_L$. Then there is a canonical isomorphism
\[
(G_0^x)^\circ / (P_0^x)^\circ \;\cong\;
\mu^{-1}(x) \cap G_0 \times^{P_0} (\mathcal{O}_L + \mathfrak{u}_n),
\]
where $\mu : G_0 \times^{P_0} (\mathcal{O}_L + \mathfrak{u}_n) \to \mathfrak{g}_n$ is the natural map.
\end{lemma}

\begin{proof}
Consider the map
\[
(G_0^x)/(P_0^x) \longrightarrow 
\mu^{-1}(x) \cap G_0 \times^{P_0} (\mathcal{O}_L + \mathfrak{u}_n),
\qquad
g P_0^x \longmapsto (g P_0, x).
\]
This map is well-defined.  
By Proposition~\ref{prop5.4}(2), $\mathcal{O}_L \subset \mathcal{O}$; hence it suffices to prove the statement for any $y \in \mathcal{O}_L$.

The injectivity is immediate.  
Let $g \in G_0$ be such that $\mathrm{Ad}(g^{-1}) y \in \mathcal{O}_L + \mathfrak{u}_n$.  
By Lemma~\ref{lm4.6}, there exists $p \in P_0$ with
\[
\mathrm{Ad}(g^{-1}) y = \mathrm{Ad}(p) y.
\]
Hence $\mathrm{Ad}(g p) y = y$, so we may replace $g$ by $g p$, yielding $g \in G_0^y$.  
Thus the map above is an isomorphism.

Since $(\mathcal{O}, \mathcal{L})$ is primitive with respect to $(P, L, \mathcal{O}_L, \mathcal{L}|_{\mathcal{O}_L})$ and $y \in \mathcal{O}_L$, the definition of a primitive pair gives
\[
G_0^y / (G_0^y)^\circ \;\cong\; P_0^y / (P_0^y)^\circ.
\]
Consequently,
\[
G_0^y / P_0^y \;\cong\; (G_0^y)^\circ / (P_0^y)^\circ,
\]
and the claim follows.
\end{proof}

\begin{remark}\label{rk5.8}
As explained at the beginning of this subsection, $(P_0^y)^\circ$ is a Borel subgroup of $(G_0^y)^\circ$.  
Therefore $(G_0^y)^\circ / (P_0^y)^\circ$ is a flag variety.  
In particular, there exist integers $s_1, \dots, s_n$ such that
\[
R\Gamma\!\left(
\underline{\Bbbk}_{\mu^{-1}(y) \cap G_0 \times^{P_0}(\mathcal{O}_L + \mathfrak{u}_n)}
\right)
\;\cong\;
\bigoplus_{j=1}^n \Bbbk[-2 s_j].
\]
\end{remark}

Before the next theorem, we assume that all $(P_0^y)^\circ$-orbits in $G_0^y / P_0^y$ are $P_0^y$-stable.

An analogue of the following theorem has appeared in \cite{Lu}. Lusztig uses the language of perverse sheaves instead of the parity complexes. Though the Theorem below follows the same pattern as of \cite{Lu} but the details are different as our sheaf coefficients are coming from a field $\Bbbk$ of positive characteristic.

\begin{theorem}\label{th4.7}
Let $(P, L, \mathcal{O}_L, \mathcal{L}|_{\mathcal{O}_L})$ be as above. Then:
\begin{enumerate}
    \item 
    \[
    \mathcal{H}^i\!\left(
    \operatorname{Ind}_{\mathfrak{p}}^{\mathfrak{g}}
    \mathcal{E}(\mathcal{O}_L, \mathcal{L}|_{\mathcal{O}_L})
    \big|_{\mathcal{O}}
    \right)
    \cong
    \begin{cases}
    0, & \text{if $i$ is odd,}\\[3pt]
    \bigoplus_j \mathcal{L}[-2 s_j], & \text{if } i = -\dim \mathcal{O};
    \end{cases}
    \]
    \item $\operatorname{Ind}_{\mathfrak{p}}^{\mathfrak{g}}
    \mathcal{E}(\mathcal{O}_L, \mathcal{L}|_{\mathcal{O}_L})
    \big|_{\mathfrak{g}_n \setminus \mathcal{O}} = 0$;
    \item
    \[
    \operatorname{Ind}_{\mathfrak{p}}^{\mathfrak{g}}
    \mathcal{E}(\mathcal{O}_L, \mathcal{L}|_{\mathcal{O}_L})
    \;\cong\;
    \bigoplus_j \mathcal{E}(\mathcal{O}, \mathcal{L})[-2 s_j];
    \]
    \item
    $\mathcal{E}(\mathcal{O}, \mathcal{L})|_{\mathfrak{g}_n \setminus \mathcal{O}} = 0$.
\end{enumerate}
\end{theorem}

\begin{proof}
\begin{enumerate}
\item
By \cite[Theorem.~26]{Ch}, the parity sheaf $\mathcal{E}(\mathcal{O}_L, \mathcal{L}|_{\mathcal{O}_L})$ exists.  
To prove the first statement, it suffices to show
\[
\operatorname{Ind}_{\mathfrak{p}}^{\mathfrak{g}}
\mathcal{E}(\mathcal{O}_L, \mathcal{L}|_{\mathcal{O}_L})
\big|_{\mathcal{O}_L}
\cong
\bigoplus_j \mathcal{L}|_{\mathcal{O}_L}[-2 s_j].
\]
Then, using the isomorphism
\[
L_0^{\pi(x)} / (L_0^{\pi(x)})^\circ
\;\cong\;
G_0^x / (G_0^x)^\circ, \qquad \forall\, x \in \mathcal{O}_L + \mathfrak{u}_n,
\]
the desired result follows.

\smallskip
\noindent
\emph{Step 1: Description of the image of $\mu$.}
We first claim that the image of
\[
\mu : G_0 \times^{P_0} (\mathcal{O}_L + \mathfrak{u}_n)
\longrightarrow \mathfrak{g}_n
\]
is $\mathcal{O}$.
Indeed, 
\[
G_0 \times^{P_0} (\mathcal{O}_L + \mathfrak{u}_n)
\simeq
\{(g P_0, x) \in G_0/P_0 \times \mathfrak{g}_n
\mid \mathrm{Ad}(g^{-1}) x \in \mathcal{O}_L + \mathfrak{u}_n\},
\]
and $\mu$ is the projection to the second coordinate.
If $x$ lies in the image, then there exists $g \in G_0$ such that
$\mathrm{Ad}(g^{-1}) x \in \mathcal{O}_L + \mathfrak{u}_n$.  
By Lemma~\ref{lm4.6}, $\mathcal{O}_L + \mathfrak{u}_n$ is a $P_0$-orbit, say $\mathcal{O}_P$.  
Hence $\mathrm{Ad}(g^{-1}) x \in \mathcal{O}_P \subset \mathcal{O}$, i.e.\ $x \in \mathcal{O}$.

\smallskip
\noindent
\emph{Step 2: Reduction to the fiber at $y \in \mathcal{O}_L$.}
Fix $y \in \mathcal{O}_L$.  
The following diagram is commutative:
\[
\begin{tikzcd}
\mathcal{O}_L 
& \mathcal{O}_L + \mathfrak{u}_n
  \arrow[l, "\pi"']
  \arrow[r, hook, "e"]
& G_0 \times^{P_0} (\mathcal{O}_L + \mathfrak{u}_n)
  \arrow[r, "\mu"]
& \mathcal{O} \\
& 
& \mu^{-1}(y) \cap G_0 \times^{P_0} (\mathcal{O}_L + \mathfrak{u}_n)
  \arrow[u, hook']
  \arrow[r]
& y \arrow[u, hook']
\end{tikzcd}
\]

One checks that
\[
G_0 \times^{P_0} (\mathcal{O}_L + \mathfrak{u}_n)
\;\cong\;
G_0 \times^{G_0^y}
(\mu^{-1}(y) \cap G_0 \times^{P_0} (\mathcal{O}_L + \mathfrak{u}_n)\bigr),
\]
via the map
\[
(g, (h P_0, y))
\longmapsto
(g h P_0, \mathrm{Ad}(g) y).
\]

Recall $\mc{O}_L+\mf{u}_n$ is a $P_0$-orbit. Hence $G_0\times^{P_0}(\mc{O}_L+\mf{u}_n)$ is isomorphic to a $G_0$-orbit. Therefore any element in $G_0\times^{P_0}(\mc{O}_L+\mf{u}_n)$ is of the form $(gP_0,y)$, whose inverse image under the above map is $(g,(P_0,y))$. Hence the above map is surjective. If $(ghP_0,\mathrm{Ad}(g)y)=(P_0,y)$, then $\mathrm{Ad}(g)y=y$. This implies $g\in G_0^y$. So we get a bijection between $G_0\times^{G_0^y} (\mu^{-1}(y)\cap G_0\times^{P_0}(\mc{O}_L+\mf{u}_n))$ and  $G_0\times^{P_0}(\mc{O}_L+\mf{u}_n)$.

Now $G_0\times^{P_0}(\mc{O}_L+\mf{u}_n)$ is smooth as $\mc{O}_L+\mf{u}_n$ is a $P_0$-orbit. Also by Lemma \ref{lm4.5}, we have $\mu^{-1}(y)\cap G_0\times^{P_0}(\mc{O}_L+\mf{u}_n)\cong (G_0^y)^\circ/(P_0^y)^\circ$. Recall  $(G_0^y)^\circ/(P_0^y)^\circ$ is a flag variety, therefore smooth. Bijection between two smooth varieties induces an isomorphism of varieties. Hence we are done with our claim. 

\smallskip
\noindent
\emph{Step 3: Induction equivalence and local systems.}

The $G_0$-equivariant local systems on $\mc{O}$, denoted by $\mathrm{Loc}_{G_0}(\mc{O})$ is isomorphic to the representation of the fundamental group, $\mathrm{Rep}(\pi_1(\mc{O}))$. 

By step 2 the inclusion,
\[
\mu^{-1}(y) \cap G_0 \times^{P_0} (\mathcal{O}_L + \mathfrak{u}_n)
\xhookrightarrow{} G_0 \times^{P_0} (\mathcal{O}_L + \mathfrak{u}_n)
\]
induces the standard induction equivalence on derived categories. The projection, \[G_0\times^{G_0^y}(\mu^{-1}(y)\cap G_0\times^{P_0}(\mc{O}_L+\mf{u}_n)) \to G_0/G_0^y \cong \mc{O}\] is a fiber bundle with the fiber $\mu^{-1}(y)\cap G_0\times^{P_0}(\mc{O}_L+\mf{u}_n) $. But being a flag variety, $\mu^{-1}(y)\cap G_0\times^{P_0}(\mc{O}_L+\mf{u}_n) $ is simply connected.

As the fiber is simply connected, we have
\[
\pi_1(\mathcal{O}) \;\cong\; \pi_1(G_0 \times^{P_0} (\mathcal{O}_L + \mathfrak{u}_n)).
\]

Now we have,

\begin{align*}
	 \mathrm{Loc}_{G_0}(\mc{O}) & \cong \mathrm{Rep}(\pi_1(\mc{O}))\\ & \cong \mathrm{Rep}(\pi_1(G_0\times^{P_0}(\mc{O}_L+\mf{u}_n))) \\ & \cong \mathrm{Loc}_{G_0^y}( \mu^{-1}(y)\cap G_0\times^{P_0}(\mc{O}_L+\mf{u}_n)) \text{, by induction equivalence.}
\end{align*} 
	On the other hand, by induction equivalence,
	\[ \mathrm{Loc}_{G_0}(\mc{O})\cong  \mathrm{Loc}_{G_0^y}(\{y\}).
	\]Combining all these we get,
	\[\mathrm{Loc}_{G_0^y}(\mu^{-1}(y)\cap G_0\times^{P_0}(\mc{O}_L+\mf{u}_n)) )\cong \mathrm{Loc}_{G_0^y}(\{y\}).
	\]Hence pullback along the map,
	\[\mu^{-1}(y)\cap G_0\times^{P_0}(\mc{O}_L+\mf{u}_n)) \to \{y\}
	\]induces an equivalence of categories.

\smallskip
\noindent
\emph{Step 4: Computation of $\mu^*\mathcal{L}$.}

	Our claim is that,
	\[ \mu^*(\mc{L})\cong (e^*\f^{G_0}_{P_0})^{-1} \pi^*(\mc{L}|_{\mc{O}_L}).
	\]
	From the diagram below,
	\[\begin{tikzcd}
		\mc{O}_L \arrow[d, bend right, "i"'] \arrow[ddr, hook']\\
		\mc{O}_L+\mf{u}_n \arrow[u, "\pi"'] \arrow[d,hook', "e"']\\
		G_0\times^{P_0}(\mc{O}_L+\mf{u}_n)  \arrow[r,"\mu"'] & \mc{O}
	\end{tikzcd}\]
Let $\mc{G}$ be a local system on $G_0\times^{P_0}(\mc{O}_L+\mf{u}_n)$, uniquely determined by, 
\[ (e^*\f^{G_0}_{P_0}) \mc{G}\cong \pi^* (\mc{L}|_{\mc{O}_L}).
\]
By Definition \ref{def4.2}, 
\[\mathrm{Loc}_{L_0}(\mc{O}_L)\cong \mathrm{Loc}_{P_0}(\mc{O}_L+\mf{u}_n),
\] and $i^*$ is the inverse of $\pi^*$.
So we can say, $\mc{G}$ is uniquely determined by,
\[  i^*(e^*\f^{G_0}_{P_0}) (\mc{G})\cong  \mc{L}|_{\mc{O}_L}.\]
From the above diagram, it is clear that,
\[ i^*(e^*\f^{G_0}_{P_0})\mu^*(\mc{L})\cong \mc{L}|_{\mc{O}_L}.
\]Hence $\mc{G}\cong \mu^*\mc{L}$.
Therefore $\inn^\mf{g}_{\mf{p}}\mc{E}(\mc{O}_L,\mc{L}|_{\mc{O}_L})$ becomes $\mu_! \mu^* \mc{L}$.
But now,
\begin{align*}
	 \mu_! \mu^* \mc{L}&\cong \mu_!(\underline {\Bbbk} \otimes^L \mu^*\mc{L})\\ & \cong \mu_!(\underline{\Bbbk})\otimes^L \mc{L}.
\end{align*}
We have the following  diagram,
\[\begin{tikzcd}
	D^b_{G_0}(G_0\times^{P_0}(\mc{O}_L+\mf{u}_n)) \arrow[d, "\cong"]\arrow[r, "\mu_!"] & D^b_{G_0}(\mc{O}) \arrow[d, "\cong"]\\ D^b_{G_0^y}(\mu^{-1}(y)\cap G_0\times^{P_0}(\mc{O}_L+\mf{u}_n)) \arrow[r, "\cong"]& D^b_{G_0^y}(y)  .
\end{tikzcd}\]
Hence $\mc{L}$ can be considered as a $G_0^y$-equivariant local system on $\{y\}$ and $\mu_!\mu^* \mc{L}$ can considered as the composition of the pullback and push-forward along the map,
\[\mu^{-1}(y)\cap G_0\times^{P_0}(\mc{O}_L+\mf{u}_n)\to \{y\}.
\] With an abuse of notation, we also call the above map to be $\mu$. As $\mu^{-1}(y)\cap G_0\times^{P_0}(\mc{O}_L+\mf{u}_n) \cong G_0^y/P_0^y $, if we can show $G_0^y/(G_0^y)^\circ$ acts on $\h^*_{G_0^y}(G_0^y/P_0^y)$ trivially, then by the discussion in Remark \ref{rk5.8}, we will be done. This is same as showing the constant local system along the map,
 \begin{equation}\label{eq5.6}
	D^b_{G_0^y}(G_0^y/P_0^y)\xrightarrow{\mu_!} D^b_{G_0^y}(\{y\}) 
\end{equation}

gets sent to a sheaf complex whose cohomologies are  constant local systems. 
By Definition \ref{def4.2},
\[\mathrm{Loc}_{G_0^y}(\{y\}) \cong \mathrm{Loc}_{P_0^y}(\{y\}).
\] Therefore (\ref{eq5.6}) reduces to showing, the map below,
\[ D^b_{P_0^y}(G_0^y/P_0^y)\to D^b_{P_0^y}(\{y\})
\]sends constant local system to a sheaf complex whose cohomologies are  constant local systems.  
Now  $G_0^y/P_0^y\cong (G_0^y)^\circ/(P_0^y)^\circ$, which is a flag variety stratified by the $(P_0^y)^\circ$-orbits, which are Schubert cells, hence affine. By our assumption, $(P_0^y)^\circ$-orbits are stable under $P_0^y$ in $G_0^y/P_0^y$. Hence our claim further reduces to showing that along the map,
\[ D^b_{P_0^y}(\mb{A}^n)\to D^b_{P_0^y}(\{y\})
\]constant local system gets sent to the constant local system, where $\mb{A}^n$ is an affine space. This is indeed true. This proves (1).

\item
As shown above, $\mathrm{Im}(\mu) = \mathcal{O}$; thus the restriction to $\mathfrak{g}_n \setminus \mathcal{O}$ vanishes, proving (2).

\item
By \cite[Theorem.~26]{Ch}, 
$\operatorname{Ind}_{\mathfrak{p}}^{\mathfrak{g}} 
\mathcal{E}(\mathcal{O}_L, \mathcal{L}|_{\mathcal{O}_L})$ 
is a parity complex.  
Combining with (1), we obtain (3).

\item
This follows directly from (1)--(3).
\end{enumerate}
\end{proof}

\section{Primitive pairs and Fourier transform}\label{sec:FT}

In this section we recall the notion of Fourier transform in the setting of graded Lie algebras and summarize its formal definition and key properties.  Let \(G\), \(\chi\), \(\mf{g} = \bigoplus_{n\in\mathbb Z}\mf{g}_n\), and \(G_0\) be as in Section~\ref{background}.

 The Fourier-Sato transform in the graded setting is the functor,  
\[
\Phi_{\mf{g}_n} : D^b_{G_0}(\mf{g}_n) \;\longrightarrow\; D^b_{G_0}(\mf{g}_{-n})
.\]    
A detailed discussion of Fourier transform can be found in \cite{Ch1}.

To keep track of the grading we use the notation ${^{\mathfrak g,\mathfrak p}}\!\mc{I}_{n}$ for the functor induction and ${^{\mathfrak g,\mathfrak p}}\!\mc{R}_{n} $ for the functor restriction.
In \cite{Ch1} the author shows that under suitable hypotheses on the coefficient field (Assumption 2.4,  Conjectures 2.8 and 2.9), this transform satisfies the following properties:

\begin{enumerate}
  \item \(\Phi_{\mf{g}_n}\) sends parity complexes in \(\mf{g}_n\) to parity complexes in \(\mf{g}_{-n}\).  
  \item \(\Phi_{\mf{g}_n}\) exchanges cuspidal pairs in \(\mf{g}_n\) with cuspidal pairs in \(\mf{g}_{-n}\).  
  \item \(\Phi_{\mf{g}_n}\) intertwines parabolic induction and restriction functors: if \(P\subset G\) is a parabolic with Levi \(L\), then one has  
  \[
    \Phi_{\mf{g}_{-n}}\circ{^{\mathfrak g,\mathfrak p}}\!\mc{I}_{n}
    \;\cong\;
    {^{\mathfrak g,\mathfrak p}}\!\mc{I}_{-n}
    \circ
    \Phi_{\mf l_n}
  \] and,
  \[\Phi_{\mf{g}_{-n}}\circ{^{\mathfrak g,\mathfrak p}}\!\mc{R}_{n}
    \;\cong\;
    {^{\mathfrak g,\mathfrak p}}\!\mc{R}_{-n}
    \circ
    \Phi_{\mf l_n}.
  \]
\end{enumerate}

The goal of this section is to study the Fourier transform of primitive pairs.
We assume that $(G,\chi)$ is $n$-rigid.

\begin{pro}\label{prop4.8}
Assume $(\mc{O},\mc{L}) \in \i'(\mf{g}_n)$ is a primitive pair associated to a parabolic $P$ containing a Levi subgroup $L$ with $(\mc{O}_L,\mc{E}_L) \in \ms{I}(\mf{l}_n)^{\cu}$, and let $\pi: \mf{p} \to \mf{l}$ be the projection map. 
Assume also that $\mc{O}_n$ and $\mc{O}_{-n}$ are the unique open $G_0$-orbits in $\mf{g}_n$ and $\mf{g}_{-n}$ respectively, and that $\mc{O}'_L$ is the unique open $L_0$-orbit in $\mf{l}_{-n}$. 
If
\[
\mc{O}_n \cap \pi^{-1}(\mc{O}_L) \ne \emptyset 
\quad \text{and} \quad 
\mc{O}_{-n} \cap \pi^{-1}(\mc{O}'_L) \ne \emptyset,
\]
then $\mc{O} = \mc{O}_n$.
\end{pro}

The above proposition is proved in \cite[Prop.~12.2]{Lu}.

\begin{theorem}\label{prop4.9}
Assume $(\mc{O},\mc{L}) \in \i'(\mf{g}_n)$, where $\mc{O}$ is the unique open $G_0$-orbit in $\mf{g}_n$. 
Suppose
\[
\Phi_{\mf{g}_n}\mc{E}(\mc{O},\mc{L})[\dim \mf{g}_n]
 \cong \mc{E}(\mc{O}',\mc{L}')[\dim \mf{g}_{-n}],
\]
with $(\mc{O}',\mc{L}') \in \i'(\mf{g}_{-n})$ and $\mc{O}'$ the unique open $G_0$-orbit in $\mf{g}_{-n}$. 
Then $(\mc{O},\mc{L})$ is a primitive pair.
\end{theorem}

\begin{proof}
By \cite[Th.~36]{Ch}, $\mc{E}(\mc{O},\mc{L})$ is normal. 
Therefore, there exists a parabolic $P$ with Levi $L$ and $(\mc{O}_L,\mc{E}_L)\in \i'(\mf{l}_n)^{\cu}$ such that $\mc{E}(\mc{O},\mc{L})$ appears as a direct summand of $\inn^{\mf{g}}_{\mf{p}} \mc{E}(\mc{O}_L,\mc{E}_L)$. 

Let $(\tilde{\mc{O}},\tilde{\mc{L}})\in \i'(\mf{g}_n)$ be the primitive pair associated to $P,L$ and $(\mc{O}_L,\mc{E}_L) \in \i'(\mf{l}_n)^{\cu}$. 
By Theorem 4.1 and Corollary 4.2 of \cite{Ch1} induction commutes with the Fourier transform. 
Hence $\phign \mc{E}(\mc{O},\mc{L})$ appears as a direct summand of $\ind_{-n} \Phi_{\mf{l}_{-n}} \mc{E}(\mc{O}_L,\mc{E}_L)$. 

From the definition of induction for a cuspidal pair \cite[Lemma~25]{Ch} and our assumptions, there exist $x \in \mc{O}$ and $g \in G_0$ such that $\mathrm{Ad}(g^{-1})x \in \mc{O}_L + \mf{u}_n$. 
The assumption $\phign \mc{E}(\mc{O},\mc{L}) = \mc{E}(\mc{O}',\mc{L}')$ implies that $\mc{E}(\mc{O}',\mc{L}')$ appears as a direct summand of $\ind_{-n} \Phi_{\mf{l}_{-n}} \mc{E}(\mc{O}_L,\mc{E}_L)$. 

By Theorem 4.5 of \cite{Ch1}, Fourier transform sends cuspidal pairs to cuspidal pairs. 
Hence there exists $(\mc{O}'_L,\mc{E}'_L) \in \i'(\mf{l}_{-n})^{\cu}$ such that 
\[
\Phi_{\mf{l}_{-n}} \mc{E}(\mc{O}_L,\mc{E}_L) = \mc{E}(\mc{O}'_L,\mc{E}'_L).
\]
This implies that $\mc{O}'_L$ is the unique open $L_0$-orbit in $\mf{l}_{-n}$. 

Thus, there exist $y \in \mc{O}'$ and $g' \in G_0$ such that $\mathrm{Ad}(g'^{-1})y \in \mc{O}'_L + \mf{u}_{-n}$. 
Replacing $x$ by $\mathrm{Ad}(g^{-1})x$ and $g$ by $1$, we see that the new $x$ belongs to both $\mc{O}$ and $\mc{O}_L + \mf{u}_n$. Therefore,
\[
\mc{O} \cap \pi^{-1}(\mc{O}_L) \neq \emptyset.
\]
Similarly, one can define $y$ such that $y \in \mc{O}' \cap (\mc{O}'_L + \mf{u}_{-n})$, and hence
\[
\mc{O}' \cap \pi^{-1}(\mc{O}'_L) \neq \emptyset.
\]
Therefore, by Proposition~\ref{prop4.8}, $\mc{O} = \tilde{\mc{O}}$. 

Now we can apply Theorem~\ref{th4.7}(1), which states that $\inn^{\mf{g}}_{\mf{p}} \mc{E}(\mc{O}_L,\mc{E}_L)$ is a direct sum of copies of $\mc{E}(\mc{O},\tilde{\mc{L}})$ up to shift. 
Combining this with our assumption, we obtain $\mc{E}(\mc{O},\mc{L}) = \mc{E}(\mc{O},\tilde{\mc{L}})$, which implies $\mc{L} = \tilde{\mc{L}}$.
\end{proof}

\section{Quasi-Monomial}\label{sec6}

In this section, we introduce \emph{good pairs}, which can be viewed as a generalization of the notion of normal complexes introduced in \cite{Lu} and \cite{Ch}.

\begin{definition}
A pair $(\mc{O},\mc{L}) \in \ms{I}(\mf{g}_n)$ is called \textbf{quasi-monomial} if there exists a parabolic subgroup $P$ with Levi subgroup $L$ such that the following hold:
\begin{enumerate}
    \item $(L,\chi)$ is $n$-rigid.
    \item There exists a primitive pair $(\mc{O}_L,\mc{E}_L) \in \ms{I}(\mf{l}_n)$ such that $\mc{E}(\mc{O},\mc{L})$ is isomorphic, up to shift, to $\inn^{\mf{g}}_{\mf{p}}\mc{E}(\mc{O}_L,\mc{E}_L)$. In this case, we refer to $\mc{E}(\mc{O},\mc{L})$ as a \textbf{quasi-monomial object}.
\end{enumerate}
\end{definition}

\begin{remark}
A quasi-monomial object is said to be \textbf{proper quasi-monomial} if $P$ is a proper parabolic subgroup of $G$.
\end{remark}

\begin{definition}
A parity complex $\mc{E}$ in $D^b_{G_0}(\mf{g}_n)$ is called a \textbf{good object} if $\mc{E}$ is a direct sum of quasi-monomial objects. It is said to be \textbf{proper} if all quasi-monomial summands are proper. 
\end{definition}

\begin{theorem}\label{th5.4}
Let $(\mc{O},\mc{L}) \in \ms{I}(\mf{g}_n)$. Then there exist good objects $\mc{E}$ and $\mc{F}$ in $D^b_{G_0}(\mf{g}_n)$ such that
\[
\mc{E}(\mc{O},\mc{L}) \oplus \mc{E}\cong \mc{F}.
\]
\end{theorem}

\begin{proof}
The proof proceeds  by induction on the reductive quotient associated to proper parabolic subgroups containing $\chi(\mathbb{C}^\times)$. The proof is divided into several steps.

\begin{enumerate}
\item 
We first show that the statement holds when $(\mc{O},\mc{L}) \in \ms{I}(\mf{g}_n)$ and either $(G,\chi)$ is not $n$-rigid, or $\mc{O}$ is not the unique open $G_0$-orbit in $\mf{g}_n$.

We argue by induction on $\dim \mc{O}$. Choose $x \in \mc{O}$ and construct $P, L, U, \mf{p}, \mf{l}, \mf{n}$ as in \cite[5.2]{Ch}. Then $P \neq G$. By \cite[Theorem ~28]{Ch}, the element $x$ lies in the open $L_0$-orbit $\mc{O}_L \subset \mf{l}_n$, and the restriction $\mc{L}|_{\mc{O}_L}$ is an irreducible local system, denoted $\mc{L}'$.

By \cite[Theorem ~30]{Ch}, the complex $\inn^{\mf{g}}_{\mf{p}}\mc{E}(\mc{O}_L,\mc{L}')$ has support $\bar{\mc{O}}$ and satisfies
\[
\inn^{\mf{g}}_{\mf{p}}\mc{E}(\mc{O}_L,\mc{L}')|_{\mc{O}} = \mc{L}[\dim \mc{O}_L].
\]
Furthermore, $\inn^{\mf{g}}_{\mf{p}}\mc{E}(\mc{O}_L,\mc{L}')$ is a parity complex by \cite[Theorem ~37]{Ch}. Hence,
\begin{equation}\label{5.6}
\inn^{\mf{g}}_{\mf{p}}\mc{E}(\mc{O}_L,\mc{L}') = \mc{E}(\mc{O},\mc{L}) \oplus \mc{E}',
\end{equation}
where $\mc{E}'$ is a parity complex supported on $\bar{\mc{O}} \setminus \mc{O}$.

By the induction hypothesis, there exist good objects $\mc{E}_1$ and $\mc{E}_2$ in $D^b_{L_0}(\mf{l}_n)$ such that
\[
\mc{E}(\mc{O}_L,\mc{L}') \oplus \mc{E}_1\cong\mc{E}_2.
\]
Applying induction yields
\begin{equation}\label{5.7}
\inn^{\mf{g}}_{\mf{p}}\mc{E}(\mc{O}_L,\mc{L}') \oplus \inn^{\mf{g}}_{\mf{p}}\mc{E}_1
\cong \inn^{\mf{g}}_{\mf{p}}\mc{E}_2.
\end{equation}
Combining (\ref{5.6}) and (\ref{5.7}) gives
\begin{equation}\label{5.8}
\mc{E}(\mc{O},\mc{L}) \oplus \mc{E}' \oplus \inn^{\mf{g}}_{\mf{p}}\mc{E}_1
\cong \inn^{\mf{g}}_{\mf{p}}\mc{E}_2.
\end{equation}

Since support of $\mc{E}'$ has smaller dimension than $\mc{O}$, by the induction hypothesis there exist good objects $\tilde{\mc{E}}'$ and $\tilde{\mc{E}}''$ such that
\begin{equation}\label{5.9}
\mc{E}' \oplus \tilde{\mc{E}}' \cong \tilde{\mc{E}}''.
\end{equation}
Combining (\ref{5.8}) and (\ref{5.9}) yields
\[
\mc{E}(\mc{O},\mc{L}) \oplus \tilde{\mc{E}}'' \oplus \inn^{\mf{g}}_{\mf{p}}\mc{E}_1
\cong \inn^{\mf{g}}_{\mf{p}}\mc{E}_2 \oplus \tilde{\mc{E}}'.
\]
By the transitivity of induction \cite[Theorem ~20]{Ch}, both $\inn^{\mf{g}}_{\mf{p}}\mc{E}_1$ and $\inn^{\mf{g}}_{\mf{p}}\mc{E}_2$ are good objects. This completes Step~(1).

\item
If $(\mc{O}',\mc{L}') \in \ms{I}(\mf{g}_{-n})$ and either $(G,\chi)$ is not $(-n)$-rigid or $\mc{O}'$ is not the unique open $G_0$-orbit in $\mf{g}_{-n}$, the statement follows by an argument parallel to Step~(1).

\item Now we show that for
a proper pair $(\mc{O},\mc{L}) \in \ms{I}(\mf{g}_{-n})$ there exist good objects $\mc{E}$ and $\mc{F}$ in $D^b_{G_0}(\mf{g}_n)$ such that
\[
\Phi_{\mf{g}_{-n}} \mc{E}(\mc{O},\mc{L})\oplus\mc{E}\cong \mc{F}.
\]
Since $\mc{E}(\mc{O},\mc{L})$ is proper, we may write
\[
\mc{E}(\mc{O},\mc{L}) = {}^{\mf{g},\mf{q}}\mc{I}_{-n}\mc{E}(\mc{O}_M,\mc{E}_M),
\]
where $Q$ is a proper parabolic subgroup of $G$ containing $\chi(\mathbb{C}^\times)$, and $(\mc{O}_M,\mc{E}_M) \in \ms{I}(\mf{m}_{-n})$ is a primitive pair with Levi $M \subset Q$.
Since induction commutes with the Fourier transform, we have
\[
\Phi_{\mf{g}_{-n}}\mc{E}(\mc{O},\mc{L}) = {}^{\mf{g},\mf{q}}\mc{I}_n \Phi_{\mf{m}_{-n}}\mc{E}(\mc{O}_M,\mc{E}_M).
\]
By induction, there exist good objects $\mc{E}'_M$ and $\mc{E}''_M$ in $\ms{I}(\mf{m}_n)$ such that
\[
\Phi_{\mf{m}_{-n}}\mc{E}(\mc{O}_M,\mc{E}_M) \oplus \mc{E}'_M \cong \mc{E}''_M.
\]
Hence,
\[
\Phi_{\mf{g}_{-n}}\mc{E}(\mc{O},\mc{L}) \oplus {}^{\mf{g},\mf{q}}\mc{I}_n\mc{E}'_M
\cong {}^{\mf{g},\mf{q}}\mc{I}_n\mc{E}''_M,
\]
and the claim follows from the transitivity of induction.

\item
Now assume $(\mc{O},\mc{L}) \in \ms{I}(\mf{g}_n)$ with $\bar{\mc{O}} = \mf{g}_n$ and the support of $(\phign\mc{E}(\mc{O},\mc{L}))$   is not the unique open $G_0$ orbit in $\mf{g}_{-n}$. By \cite[Theorem ~4.5]{Ch1},  $\phign\mc{E}(\mc{O},\mc{L})$ is a parity sheaf. By Step~(2), there exist good objects $\mc{E}'$ and $\mc{E}''$ in $D^b_{G_0}(\mf{g}_{-n})$ such that
\[
\phign\mc{E}(\mc{O},\mc{L}) \oplus \mc{E}' \cong \mc{E}''.
\]
Applying $\Phi_{\mf{g}_n}$ one more time we get,
\[
\mc{E}(\mc{O},\mc{L}) \oplus \Phi_{\mf{g}_{-n}}\mc{E}' \cong \Phi_{\mf{g}_{-n}}\mc{E}''.
\]
Using Step~(3) for $\mc{E}'$ and $\mc{E}''$ yields the result.

\item
Finally, if $(G,\chi)$ is not $n$-rigid or $\mc{O}$ is not the open $G_0$-orbit, the result follows from Step~(1). If the assumptions of Step~(4) hold, the result follows from that step as well. Therefore, we may assume $(G,\chi)$ is $n$-rigid, $\bar{\mc{O}} = \mf{g}_n$, and the support of $\phign\mc{E}(\mc{O},\mc{L})$ is not unique open $G_0$-orbit in $\mf{g}_{-n}$. By Theorem~\ref{prop4.9}, such $(\mc{O},\mc{L})$ is primitive, hence good.
\end{enumerate}
\end{proof}

\begin{remark}
An analogous statement to Theorem~\ref{th5.4} holds for $\mf{g}_{-n}$, and the proof proceeds identically.
\end{remark}

\end{document}